\documentclass[12pt]{amsart}

\usepackage{amsmath}
\usepackage{amssymb}
\usepackage{fullpage}
\usepackage{array}
\usepackage{enumitem}
\usepackage{verbatim} 
\usepackage{hyperref}

\usepackage{epic}
\usepackage{eepic}

\usepackage{pstricks-add}

\theoremstyle{plain}
\newtheorem{theorem}{Theorem}[section]
\newtheorem{lemma}[theorem]{Lemma}

\newtheorem{definition}[theorem]{Definition}
\theoremstyle{remark}

\newtheorem*{remarks}{Remarks}

\numberwithin{equation}{section}

\newcommand{\eqnlabel}[1]{\label{eqn:#1}}   

\newcommand{\eqnref}[1]{\eqref{eqn:#1}} 

\newcommand{\bK}{\mathbb{K}}
\newcommand{\bP}{\mathbb{P}}
\newcommand{\bR}{\mathbb{R}}
\newcommand{\bC}{\mathbb{C}}

\newcommand{\bH}{\mathbb{H}}
\newcommand{\bO}{\mathbb{O}}

\newcommand{\cX}{\mathcal{X}}

\newcommand{\Dil}{\mathrm{Dil}}

\newcommand{\pr}{\mathrm{pr}}

\newcommand{\Bigsetof}[2]{\begin{Bmatrix} #1 \,\Big|\, #2 \end{Bmatrix}}
\newcommand{\myand}{\enspace\text{and}\enspace}

\newcommand{\inv}{^{-1}}

\newcommand{\msk}{\medskip}
\newcommand{\ssk}{\smallskip}
\newcommand{\nin}{\noindent}

\begin{document}

\title{Torsors and ternary Moufang loops arising in projective geometry}

\author{Wolfgang Bertram}

\address{Institut \'{E}lie Cartan Nancy \\
Universit\'{e} de Lorraine, CNRS, INRIA \\
Boulevard des Aiguillettes, B.P. 239 \\
F-54506 Vand\oe{}uvre-l\`{e}s-Nancy, France}

\email{\url{wolfgang.bertram@univ-lorraine.fr}}

\author{Michael Kinyon}

\address{Department of Mathematics \\
University of Denver \\
2360 S Gaylord St \\
Denver, Colorado 80208 USA}

\email{\url{mkinyon@du.edu}}



\begin{abstract}
We  give an interpretation of the construction of torsors from
\cite{BeKi10} in terms of classical projective geometry.
For the Desarguesian case, this leads to a reformulation of certain results from
\cite{BeKi10}, whereas for the Moufang case the result is new.
But even in the Desarguesian case it sheds new light on the relation between
the lattice structure and the algebraic structures of a projective space.
\end{abstract}

\maketitle

\section{The geometric construction}

\subsection{The generic case}
In this first chapter, we describe the general construction of torsors and of
ternary loops associated to projective spaces; proofs and computational descriptions
are given in the two following chapters.
We assume  that $\cX$ is a projective space
of dimension at least two. For projective subspaces $a,b$ of $\cX$,
let as usual $a \land b$ be the meet (intersection) and $a \lor b$
be the join (smallest subspace containing $a$ and $b$).

\begin{definition}\label{GenericDefinition}
Consider a pair  $(a,b)$  of hyperplanes in $\cX$ and a triple of points
$(x,y,z)$, none of them in $a$ or $b$. Assume that $x,y,z$ are not collinear.
Then we define a fourth point $w:= xz :=  (xyz)_{ab}$ as follows:
$w$ is the intersection of

\ssk
-- the parallel of the line $x \lor y$ through $z$ in the affine space
$V_a:=\cX \setminus a$, with

-- the parallel of the line $z \lor y$ through $x$ in the affine space
$V_b= \cX \setminus b$; that is:
\[
w = xz= (xyz)_{ab} =  \Bigl( \bigl( (x \lor y) \land a \bigr) \lor z \Bigr)
\land
\Bigl( \bigl( ( z \lor y) \land b \bigr) \lor x \Bigr) \ .
\]
\end{definition}

\nin
Note that this point of intersection exists since all lines belong to
the projective plane spanned by $x,y,z$. For $a=b$, this is the usual
``parallelogram definition'' of vector addition in the affine space
$V_a$ with origin $y$, that is, $xz = x + z$ in this case. Hence, for
$a \not= b$, $(xyz)_{ab}$  may be seen as a kind of ``deformation of
vector addition'': we have  a sort of ``fake parallelogram'' with
vertices $y,x,z,w$, as shown in the illustrations below.
As for ``usual'' parallelograms, it is easily seen  that, with
$(xyz):=(xyz)_{ab}$ for fixed $(a,b)$, the conditions
\begin{equation}
w = (xyz)\,, \qquad
y=(zwx)\,, \qquad
z=(yxw)\,,\qquad
x=(wzy)
\end{equation}
are all equivalent.
Note also  that it is obvious from the definition that
\begin{equation}
(xyz)_{ba}=(zyx)_{ab} \, .
\end{equation}
If we represent $a$ and $b$ by affine lines, intersecting in the
affine drawing plane, the construction is visualized like this:

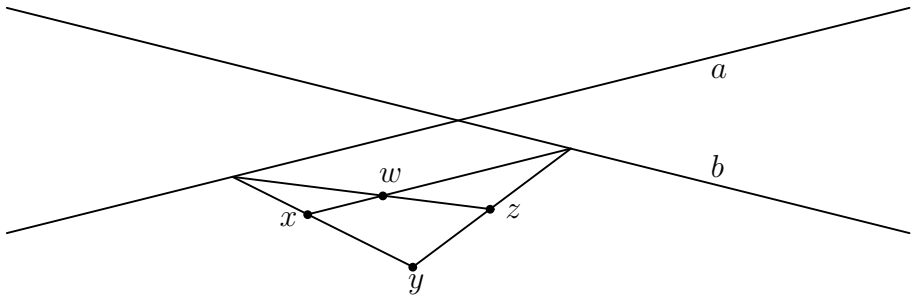
\begin{figure}[htbp]
\begin{center}
\setlength{\unitlength}{0.3mm}
\begin{picture}(400,100)
    \allinethickness{0.254mm}\path(0,-100)(400,0)
    \put(312,-31){\shortstack{$a$}}
    \allinethickness{0.254mm}\path(0,-0)(400,-100)
    \put(312,-75){\shortstack{$b$}}
    \allinethickness{0.254mm}\path(100,-75)(180,-115)
    \allinethickness{0.254mm}\path(250,-62.5)(180,-115)
    \allinethickness{0.254mm}\path(250,-62.5)(133.33,-91.67)
    \allinethickness{0.254mm}\path(100,-75)(214.286,-89.2857)
    \put(133.33,-91.67){\ellipse*{3}{3}}
    \put(121,-97){\shortstack{$x$}}
    \put(166.667,-83.333){\ellipse*{3}{3}}
    \put(165,-77){\shortstack{$w$}}
    \put(214.286,-89.2857){\ellipse*{3}{3}}
    \put(221,-93){\shortstack{$z$}}
    \put(180,-115){\ellipse*{3}{3}}
    \put(178,-124){\shortstack{$y$}}
\end{picture}
\end{center}
\vspace{4cm}
\end{figure}

\nin If we choose $a$ as ``line at infinity'' of our drawing plane,
and if we choose to draw $b$ horizontally, then we get the following
image:

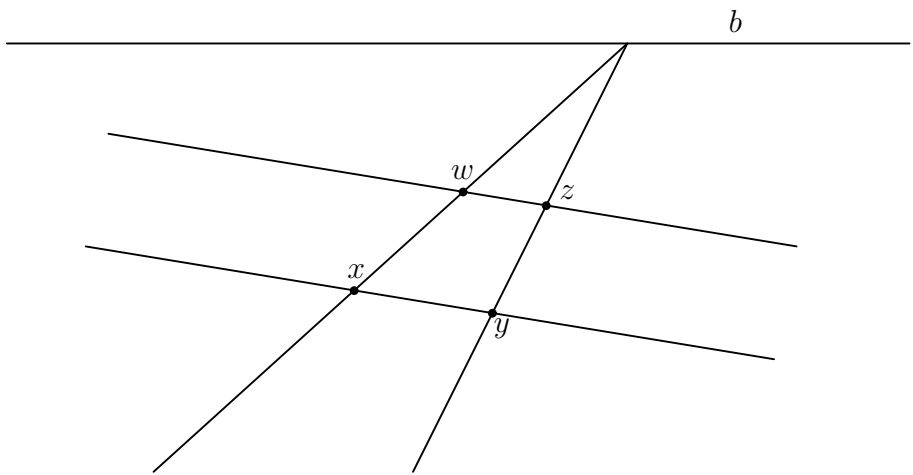
\begin{figure}[htbp]
\vspace{-3cm}
\begin{center}
\setlength{\unitlength}{0.3mm}
\begin{picture}(400,100)
    \allinethickness{0.254mm}\path(0,-10)(400,-10)
    \put(320,-5){\shortstack{$b$}}
    \allinethickness{0.254mm}\path(275,-10)(180,-200)
    \allinethickness{0.254mm}\path(275,-10)(65,-200)
    \allinethickness{0.254mm}\path(45,-50)(350,-100)
    \allinethickness{0.254mm}\path(35,-100)(340,-150)
    \put(153.970,-119.503){\ellipse*{3}{3}}
    \put(151,-114){\shortstack{$x$}}
    \put(202.290,-75.7852){\ellipse*{3}{3}}
    \put(197,-70){\shortstack{$w$}}
    \put(239.091,-81.8182){\ellipse*{3}{3}}
    \put(245,-79){\shortstack{$z$}}
    \put(215.227,-129.545){\ellipse*{3}{3}}
    \put(216,-138){\shortstack{$y$}}
\end{picture}
\end{center}
\vspace{6cm}
\end{figure}

\nin These images admit a spacial interpretation: we may imagine the
observer placed in affine space $\bR^3$ inside a plane $B$ which is vizualized only
by its ``horizon'', the line $b$; then we think of the line $y \lor z$ as lying
in a plane $B'$ parallel to $B$, and of the line $x \lor w$
as lying in another such plane $B''$; the other two lines $w\lor z$
and $x \lor y$ lie in planes that are parallel to the drawing plane $P$.
This interpretation is not symmetric in $x$ and $z$: 
the point $z$ lies  ``behind'' (or ``in front of'') $y$, whereas $x$ is considered to
be ``on the same level'' as $y$. 

\ssk
The product $xz$ is thus in general not commutative, but it is
\emph{associative}: we  show that,
if $\cX$ is Desarguesian, then, for any fixed origin $y$, the binary map
$(x,z) \mapsto xz$ gives rise to a \emph{group law}
on the intersection of affine parts
\begin{equation}\label{U_ab-def}
U_{ab} := \cX \setminus (a \cup b) = V_a \cap V_b \ .
\end{equation}
More generally and more conceptually, we show that the ternary law
$(x,y,z) \mapsto (xyz)_{ab}$ defines a \emph{torsor
structure on} $U_{ab}$ (Theorem \ref{TorsorDesarguesTheorem}).
Naturally, the question arises what we can say for general, non-Desarguesian
projective planes, or for still  more general lattices.
The most prominent class of non-Desarguesian projective planes are the
\emph{Moufang planes}: we show that in this case we get a kind of
``alternative version of a torsor'' which we call a \emph{ternary Moufang loop}
(Theorem \ref{MoufangTheorem}).
For $a=b$, these ternary Moufang loops contract to the abelian vector group
of an affine plane. For very general projective planes (which need not be
``translation planes'') it remains an interesting open problem to relate
this new algebraic structure to those traditionally considered in the
literature: indeed, our definition is closely related to the more traditional
ways of coordinatizing projective planes by \emph{ternary rings}.
This is related to the following item.

\subsection{The collinear case}
We have not yet defined what $(xyz)_{ab}$ should mean if $x,y,z$ are
\emph{collinear}. If $\cX$ is a \emph{topological} projective plane, then
one would like to complete our definition simply ``by continuity'', \emph{e.g.},
by taking the limit of $(xyz)_{ab}$ as $y$, not lying on the line $x \lor z$,
 converges to a point on $x \lor z$.
This is indeed what happens in the classical planes over the division algebras
$\bR,\bC,\bH,\bO$.
Since we do not know whether in very general cases such a ``limit'' exists,
we restrict ourselves here to the Moufang case, and leave the general case for
later work.

\begin{definition}\label{CollinearDefinition}
Assume that $\cX$ is a Moufang plane or a projective space of dimension
bigger than $2$. Consider  a pair $(a,b)$ of hyperplanes and a
\emph{collinear}  triple $(x,y,z)$ of points, none of them in $a$ or $b$.
\begin{enumerate}
\item
If $x=y=z$, let $(xyz)_{ab}:=x$.
\item
If $x \not=y$, then let us
choose a point $u$ not belonging to $x \lor y$ or to  $a$, and we let
$w:= (xyz)_{ab} : = $
\[
\qquad
(x \lor y)  \land  \Bigl[  \Bigl( ( z \lor u) \land b \Bigr) \lor
\Bigl(  \Bigl[ \bigl( (x \lor y) \land a \bigr) \lor u \Bigr]
\land
\Bigl[ \bigl( ( u \lor y) \land b \bigr) \lor x \Bigr] \Bigr)  \Bigr] \ .
\]
(It will be shown below that $w$ does not depend on the choice of $u$.)
\item
If  $z \not= y$, then we let
\[
w:= (xyz)_{ab} : = (zyx)_{ba} \,,
\]
where the right hand side is defined by the preceding case.
\end{enumerate}
\end{definition}

\nin This definition can be interpreted from two different viewpoints:

\ssk
(A) \emph{Algebraic.}  In the Desarguesian case, the expression in (2) is
derived from our first definition by using para-associativity and idempotency:
\[
((xyu)uz) = (xy(uuz)) = (xyz)\,,
\]
where now the left hand side can be expressed by using twice Definition
\ref{GenericDefinition} (see Theorem \ref{LinkGammaLattice} below).
This is indeed in keeping with idea explained above of ``taking a limit''
(imagine $u$ tending towards a point on the line $L$).
The argument still goes through in the Moufang case since one does not need for it
full para-associativity, but just a special case which remains valid precisely in the
Moufang case (but it breaks down as soon as one wants to go further).

\msk
(B) \emph{Geometric.}
The  formula in (2)  corresponds to classical ``constructions of the field
associated to a plane''. It is known that in the Moufang case the field does
not depend on the ``off-line'' point $u$. More specifically, we distinguish
two cases in (2):
\begin{itemize}
\item
generic case:
if the points $L \land a$ and $L\land b$ are different, then 
$(xyz)_{ab}$  is the product $zy\inv x$ on the vector line $L$
with ``point at infinity'' $L \land a$ and  ``zero point'' $L \land b$
(in the following illustration, $a$ is the line at infinity and $b$ the horizontal line;
in usual textbook drawings, the inverse choice is made. We have marked the points
$p=(xyu)$ and $q=(puz)=w$.)
\end{itemize}

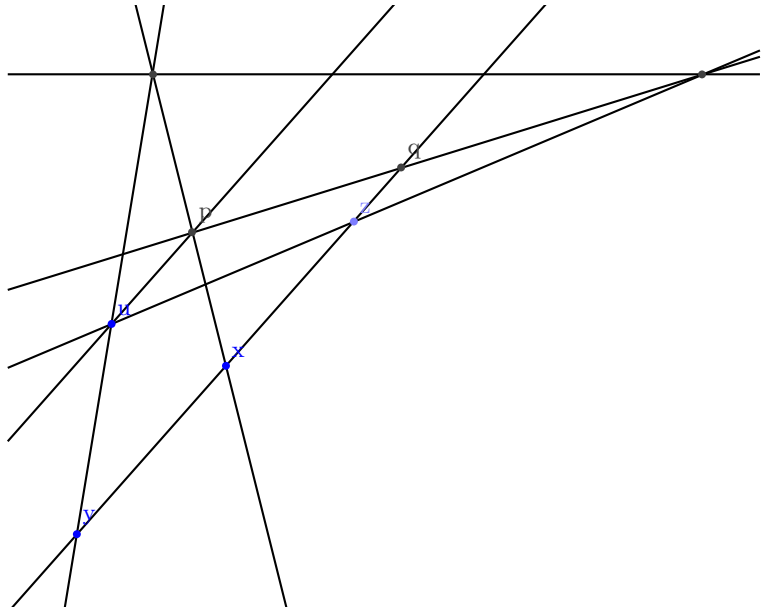
\begin{center}
\newrgbcolor{xdxdff}{0.49 0.49 1}
\psset{xunit=1.0cm,yunit=1.0cm,algebraic=true,dotstyle=o,dotsize=3pt 0,linewidth=0.8pt,arrowsize=3pt 2,arrowinset=0.25}
\begin{pspicture*}(2,-3)(12,5)
\psplot{2}{12}{(--71.89-0*x)/17.62}
\psplot{2}{12}{(-10.58--2.24*x)/1.98}
\psplot{2}{12}{(-6.07--2.24*x)/1.98}
\psplot{2}{12}{(--9.11-2.8*x)/-0.46}
\psplot{2}{12}{(--19.21-3.88*x)/0.97}
\psplot{2}{12}{(-2.16--1.36*x)/3.22}
\psplot{2}{12}{(-4-2.1*x)/-6.77}
\begin{scriptsize}
\psdots[dotstyle=*,linecolor=blue](-1.34,4.08)
\psdots[dotstyle=*,linecolor=blue](16.28,4.08)
\psdots[dotstyle=*,linecolor=blue](2.92,-2.04)
\rput[bl](3,-1.92){\blue{y}}
\psdots[dotstyle=*,linecolor=blue](4.9,0.2)
\rput[bl](4.98,0.32){\blue{x}}
\psdots[dotstyle=*,linecolor=xdxdff](6.6,2.12)
\rput[bl](6.68,2.24){\xdxdff{z}}
\psdots[dotstyle=*,linecolor=blue](3.38,0.76)
\rput[bl](3.46,0.88){\blue{u}}
\psdots[dotstyle=*,linecolor=darkgray](3.93,4.08)
\psdots[dotstyle=*,linecolor=darkgray](4.45,1.98)
\rput[bl](4.54,2.1){\darkgray{p}}
\psdots[dotstyle=*,linecolor=darkgray](11.23,4.08)
\psdots[dotstyle=*,linecolor=darkgray](7.23,2.84)
\rput[bl](7.32,2.96){\darkgray{q}}
\end{scriptsize}
\end{pspicture*}
\end{center}

\begin{itemize}
\item
special case:
if the line $L:=x \lor y$ intersects $a \land b$, then
$(xyz)_{ab}$  is the ``ternary sum''  $x-y+z$ in the affine line
$L$ (with  $L \land (a \land b)$ as
``point at infinity'', see the following illustration, which is  the limit case of the
preceding one, as $L$ becomes parallell to $b$).
\end{itemize}

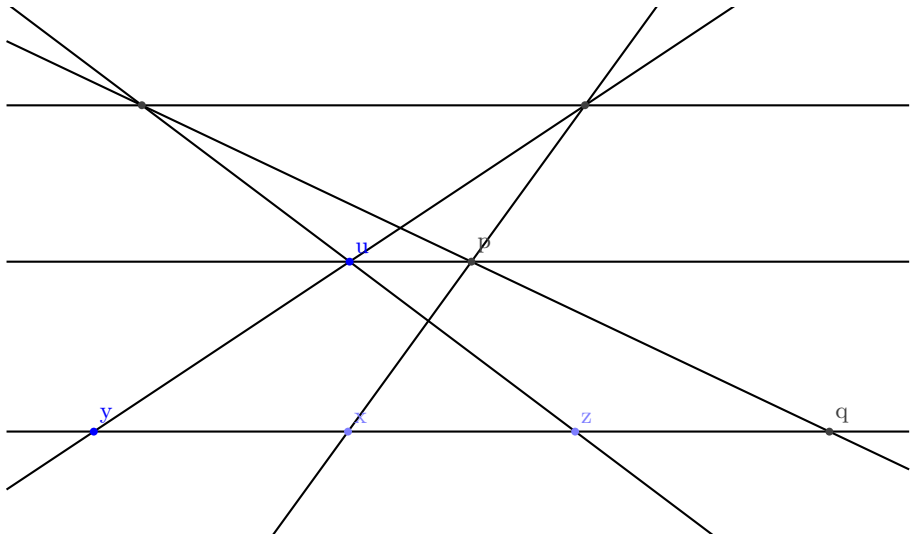
\begin{center}
\newrgbcolor{xdxdff}{0.49 0.49 1}
\psset{xunit=1.0cm,yunit=1.0cm,algebraic=true,dotstyle=o,dotsize=3pt 0,linewidth=0.8pt,arrowsize=3pt 2,arrowinset=0.25}
\begin{pspicture*}(1,-2)(13,5)
\psplot{1}{13}{(--57.65-0*x)/15.58}
\psplot{1}{13}{(--25.24-0*x)/15.58}
\psplot{1}{13}{(-9.97-0*x)/15.58}
\psplot{1}{13}{(-7.06--2.26*x)/3.4}
\psplot{1}{13}{(-17.43--2.26*x)/-3}
\psplot{1}{13}{(--26.06-4.34*x)/-3.15}
\psplot{1}{13}{(--22.03-2.08*x)/4.38}
\begin{scriptsize}
\psdots[dotstyle=*,linecolor=blue](5.56,1.62)
\rput[bl](5.64,1.74){\blue{u}}
\psdots[dotstyle=*,linecolor=blue](2.16,-0.64)
\rput[bl](2.24,-0.52){\blue{y}}
\psdots[dotstyle=*,linecolor=xdxdff](5.54,-0.64)
\rput[bl](5.62,-0.52){\xdxdff{x}}
\psdots[dotstyle=*,linecolor=xdxdff](8.56,-0.64)
\rput[bl](8.64,-0.52){\xdxdff{z}}
\psdots[dotstyle=*,linecolor=darkgray](2.8,3.7)
\psdots[dotstyle=*,linecolor=darkgray](8.69,3.7)
\psdots[dotstyle=*,linecolor=darkgray](7.18,1.62)
\rput[bl](7.26,1.74){\darkgray{p}}
\psdots[dotstyle=*,linecolor=darkgray](11.94,-0.64)
\rput[bl](12.02,-0.52){\darkgray{q}}
\end{scriptsize}
\end{pspicture*}
\end{center}

\nin The main result of the present work can now be stated as follows:

\begin{definition}\label{TorsorDefinition}
A  set $G$  with a map
$G^3 \to G$, $(x,y,z) \mapsto (xyz)$ is called a \emph{torsor} if
\begin{align*}
 (xxy)&=y=(yxx) \tag{T0} \\
 (xy(zuv))& = (x(uzy)v)=((xyz)uv) \tag{T1}
 \end{align*}
 and a \emph{ternary Moufang loop} if it  satisfies  \emph{(T0)} and
\begin{align*}
(uv(xyx)) &= ((uvx)yx) \tag{MT1} \\ 
(xy(xyz)) &= ((xyx)yz)  \tag{MT2}   
\end{align*}
\end{definition}

\begin{theorem}
If $\cX$ is a projective space of dimension bigger than one
over a skew-field (i.e., a Desarguesian space),
then the preceding constructions define a torsor law on $U_{ab}$.
If $\cX$ is a Moufang projective plane, then the constructions define a
ternary Moufang loop.
\end{theorem}

\subsection{Generalized cross-ratios, and associative geometries}
In the Desarguesian case, a very general theory describing torsors of the kind of $U_{ab}$
 has been developed in \cite{BeKi10}.  Comparing with the approach
presented here, one may ask for what kinds of lattices there are similar
theories -- we will, in subsequent work, investigate in more depth the case
of Moufang spaces, related to \emph{alternative} algebras, triple systems and
pairs. Returning to the Desarguesian case and to classical projective geometry,
the link between the lattice and the structure defined in \cite{BeKi10} is
surprisingly close; however, one should not forget that for \emph{projective
lines} the lattice structure is completely useless, whereas the structures from
\cite{BeKi10} are at least as strong as the classical \emph{cross-ratio}, and
hence are much stronger than the lattice structure. Let us briefly explain this.
Given  a unital ring $\bK$ and $\cX:=\cX(\Omega)$,  the full Grassmannian
geometry of some $\bK$-module $\Omega$ (set of all submodules of $\Omega$),
we have associated  in \cite{BeKi10} to any $5$-tuple $(x,a,y,b,z) \in \cX^5$
another element of $\cX$ by
\begin{equation}\label{GammaDef}
\Gamma(x,a,y,b,z)  :=
\Bigsetof{\omega \in \Omega}
{\begin{array}{c}
\exists \xi \in x,
\exists \alpha \in a,
\exists \eta \in y,
\exists \beta \in b,
\exists \zeta \in z : \\
\omega = \zeta + \alpha
= \alpha + \eta + \beta
= \xi + \beta
\end{array}}\, .
\end{equation}
In \cite{BeKi10}, Theorem 2.4, it is shown that the lattice structure is
recovered via
\begin{equation}
x \land a = \Gamma(x,a,y,x,a), \qquad
b \lor a = \Gamma(a,a,y,b,b)
\end{equation}
for any $y \in \cX$.
On the other hand, in the present work we prove
(Theorem \ref{LinkGammaLattice}) that, if $\bK$ is a field, if $a,b$ are
hyperplanes and $x,y,z$ one-dimensional subspaces, then $\Gamma(x,a,y,b,z)$
can be recovered from the lattice structure via
\begin{equation}
\Gamma(x,a,y,b,z)= (xyz)_{ab}\,.
\end{equation}
Thus, roughly speaking, for Desarguesian projective spaces of dimension
bigger than one, $\Gamma$ and the lattice structure are essentially equivalent
data.
Summing up, there are two major approaches to our object:
the algebraic approach (\cite{BeKi10}), based on associative algebras and -pairs and on an
underlying group structure of the ``background'' $\Omega$ (cf.\ \cite{Be12}),
and the lattice theoretic approach from the present work, keeping close to classical geometric language,
and paving the way to incorporate exceptional geometries into the picture.

\section{The Desarguesian case}

\begin{theorem}\label{TorsorDesarguesTheorem}
Assume $\cX$ is a Desarguesian projective space of dimension bigger than one, and fix a pair
$(a,b)$ of hyperplanes. Then $U_{ab}$, together with the ternary product
$(xyz):=(xyz)_{ab}$ defined above, is a torsor.
 In particular, if we  fix an ``origin'' $y \in U_{ab}$, then
$U_{ab}$ with product $xz=(xyz)_{ab} $ and origin $y$ becomes
a  group. If $a \not= b$, then this is group is not commutative.
\end{theorem}

\begin{proof} We give three different proofs. In all cases, the main
point is to prove ``para-associativity'' (T1).

\ssk
(a) The first proof is  rather a ``drawing exercise'':
let us show, just by using Desargues' Theorem, that $((xyz)uv)=(xy(zuv))$.
We construct first the point $((xyz)uv)$.
This is best visualized by choosing $a$ as line at infinity of our drawing plane,
and we may draw the lines $y \lor x$ and $z \lor (xyz)$ as vertical lines.
Then $((xyz)uv)$ is the point $q$ in the following illustration:

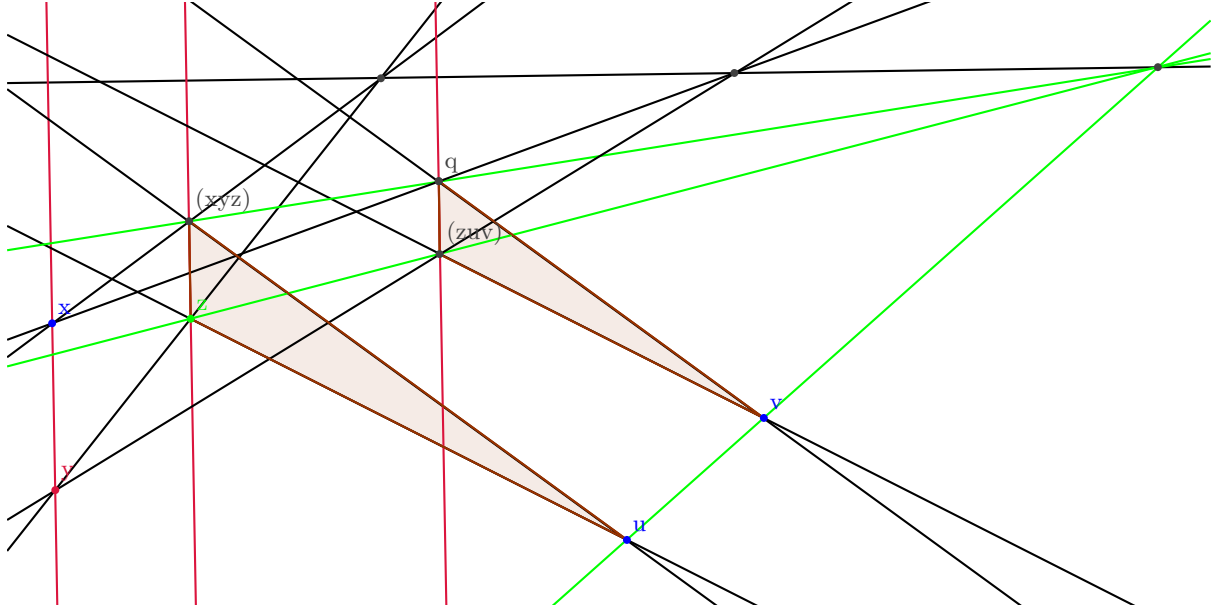
\begin{center}
\newrgbcolor{dcrutc}{0.86 0.08 0.24}
\newrgbcolor{zzttqq}{0.6 0.2 0}
\psset{xunit=1.0cm,yunit=1.0cm,algebraic=true,dotstyle=o,dotsize=3pt 0,linewidth=0.8pt,arrowsize=3pt 2,arrowinset=0.25}
\begin{pspicture*}(4,-3)(20,5)
\pspolygon[linecolor=zzttqq,fillcolor=zzttqq,fillstyle=solid,opacity=0.1](6.42,2.1)(6.44,0.8)(12.24,-2.14)
\pspolygon[linecolor=zzttqq,fillcolor=zzttqq,fillstyle=solid,opacity=0.1](9.75,1.66)(9.74,2.63)(14.06,-0.52)
\psplot{4}{20}{(--95.87--0.34*x)/24.7}
\psplot[linecolor=dcrutc]{4}{20}{(--10.24-2.22*x)/0.04}
\psplot{4}{20}{(--13.24-2.28*x)/-1.8}
\psplot{4}{20}{(-11.78--3.26*x)/4.37}
\psplot[linecolor=dcrutc]{4}{20}{(--14.33-2.22*x)/0.04}
\psplot[linecolor=green]{4}{20}{(-23.72--1.62*x)/1.82}
\psplot{4}{20}{(--23.57-2.94*x)/5.8}
\psplot{4}{20}{(--39.4-4.24*x)/5.82}
\psplot{4}{20}{(--56.55-4.24*x)/5.82}
\psplot{4}{20}{(--38.32-2.94*x)/5.8}
\psplot[linecolor=green]{4}{20}{(--13.87--2.05*x)/12.89}
\psplot[linecolor=green]{4}{20}{(-11.26--3.35*x)/12.86}
\psplot[linecolor=dcrutc]{4}{20}{(--1243.15-127.05*x)/2.29}
\psplot{4}{20}{(--22.15-3.14*x)/-5.11}
\psplot{4}{20}{(-8.6--3.33*x)/9.07}
\psline[linecolor=zzttqq](6.42,2.1)(6.44,0.8)
\psline[linecolor=zzttqq](6.44,0.8)(12.24,-2.14)
\psline[linecolor=zzttqq](12.24,-2.14)(6.42,2.1)
\psline[linecolor=zzttqq](9.75,1.66)(9.74,2.63)
\psline[linecolor=zzttqq](9.74,2.63)(14.06,-0.52)
\psline[linecolor=zzttqq](14.06,-0.52)(9.75,1.66)
\begin{scriptsize}
\psdots[dotstyle=*,linecolor=blue](-4.46,3.82)
\rput[bl](-4.24,4){\blue{$A$}}
\psdots[dotstyle=*,linecolor=blue](20.24,4.16)
\rput[bl](20,4.28){\blue{$B$}}
\psdots[dotstyle=*,linecolor=blue](4.6,0.74)
\rput[bl](4.68,0.86){\blue{x}}
\psdots[dotstyle=*,linecolor=dcrutc](4.64,-1.48)
\rput[bl](4.72,-1.36){\dcrutc{y}}
\psdots[dotstyle=*,linecolor=green](6.44,0.8)
\rput[bl](6.52,0.92){\green{z}}
\psdots[dotstyle=*,linecolor=darkgray](8.97,4)
\psdots[dotstyle=*,linecolor=darkgray](6.42,2.1)
\rput[bl](6.5,2.22){\darkgray{(xyz)}}
\psdots[dotstyle=*,linecolor=blue](12.24,-2.14)
\rput[bl](12.32,-2.02){\blue{u}}
\psdots[dotstyle=*,linecolor=blue](14.06,-0.52)
\rput[bl](14.14,-0.4){\blue{v}}
\psdots[dotstyle=*,linecolor=darkgray](19.3,4.15)
\psdots[dotstyle=*,linecolor=darkgray](9.74,2.63)
\rput[bl](9.82,2.74){\darkgray{q}}
\psdots[dotstyle=*,linecolor=darkgray](9.75,1.66)
\rput[bl](9.84,1.78){\darkgray{(zuv)}}
\psdots[dotstyle=*,linecolor=darkgray](13.67,4.07)
\psdots[dotstyle=*,linecolor=darkgray](9.74,2.63)
\end{scriptsize}
\end{pspicture*}
\end{center}

\nin
Now, the triangles $u,z,(xyz)$ and $v,(zuv),q$ are in a Desargues
configuration, and we conclude that the line $q \lor (zuv)$ is parallel to
$z \lor (xyz)$, i.e., it is vertical. But then the triangles $y,z,(zuv)$
and $x,(xyz),q$ are also in Desargues configuration, \emph{i.e.}, the
intersection points of corresponding sides lie on a  common line, which
must be $b$. It follows that $(x \lor q)\land b = (y \lor (zuv))\land b$,
from which the desired equality follows.

As a next challenge, one may try to establish the remaining equality defining
para-associativity  in the same vein.

\ssk
(b) A computational proof.  Let $\bK$ be the (skew)field of $\cX$, and work
in the affine space $V:= V_a$. If $a=b$, then (as mentioned above),
$(xyz)_{aa} = x - y +z$ is the torsor law of the affine space $V_a$, and
the claim is obviously true.
If $a \not= b$, fix some arbitrary origin $o$ in the affine hyperplane
$V_a \cap b$. There is a linear form $\beta:  V \to \bK$ such that
$b \cap V =\ker(\beta)$, so that
$U_{ab} = \{ x \in V_a \mid \beta(x) \in \bK^\times \}$.

\begin{lemma}\label{ComputationLemma}
For all $x,y,z \in U_{ab}$, in the vector space $(V_a,o)$, we have
\[
(xyz)_{ab} = \beta(z)\beta(y)\inv (x-y) + z .
\]
\end{lemma}

\begin{proof}
Assume first that $x,y,z$ are not collinear. The parallel of $x\lor y$
through $z$ in $V_a$ is
\[
 \bigl( (x \lor y) \land a \bigr) \lor z = \{ z + s (x-y) \mid \, s \in \bK \} .
\]
We determine the point $(y\lor z)\land b$. If $y\lor z$ is parallel to $b$,
then we get easily from the definition that $(xyz)_{ab} = x-y+z$ is the usual
sum, which is in keeping with our claim. Assume that $y \lor z$ is not parallel
to $b$. Then the intersection point $(y\lor z) \land $
is obtaining by solving
$\beta\bigl( (1-t)y + t z \bigr) =0$, whence
$t = \beta(y) (\beta(y) - \beta(z))\inv$, whence
$1-t = - \beta(z) (\beta(y) - \beta(z))\inv$ and
\[
(y\lor z) \land b =
- \beta(z) (\beta(y) - \beta(z))\inv y + \beta(y) (\beta(y) - \beta(z))\inv \ .
\]
The intersection of
$\bigl( ( z \lor y) \land b \bigr) \lor x$ and $\bigl( (x \lor y) \land a \bigr) \lor z = z + \bK (x-y)$
is determined by  $r,s \in \bK$ such that
\[
(1-r) x + r \bigl( - \beta(z) (\beta(y) - \beta(z))\inv y + \beta(y) (\beta(y) - \beta(z))\inv  z \bigr)
= s x - sy + z .
\]
Since both sides are barycentric combinations of $x,y,z$, we may consider $y$
as new origin. Then, if $x$ and $z$ are linearly independent with respect to
this origin, this condition is equivalent to
\[
1-r = s, \qquad  r ( \beta(y) \bigl(\beta(y) - \beta(z) \bigr)\inv ) = 1 \,
\]
whence $r = \bigl(\beta(y) - \beta(z) \bigr) \beta(y)\inv$ and
$s=1-r= \beta(z) \beta(y)\inv$, and finally
\[
(xyz)_{ab} = s (x-y) + z = \beta(z) \beta(y)\inv \bigl( x - y \bigr) + z \,,
\]
proving our claim in the non-collinear case.

Now consider the collinear case.
As pointed out after Definition \ref{CollinearDefinition}, in this case
the definition of $(xyz)_{ab}$ amounts to the geometric definition of
the field operations. If the line $L$ spanned by $x,y,z$ is parallel
to $b$, then $\beta(z)=\beta(y)$, and the formula from the lemma gives
 the additive torsor law $x-y+x$, as required.
Else, choose $o:= L \land b$ as origin, let $u \in L$ with $\beta (u)=1$ and
write $x = \xi u$, $y=\eta u$, $z=\zeta u$ with
$\xi,\eta,\zeta \in \bK^\times$,
and then the formula from the lemma gives
$(xyz)_{ab} = \zeta \eta\inv( \xi u - \eta u) + \zeta u = \zeta \eta\inv \xi u$,
which again corresponds to the definition given in this case.
Thus the claim holds in all cases.
\end{proof}

Using the lemma, we now prove the torsor laws:
first of all, we have
\begin{equation}\label{TorsorFormula2}
\beta \bigl( (xyz) \bigr) =
\beta \Bigl(  \beta(z) \beta(y)\inv \bigl( x - y \bigr) + z   \Bigr) =
\beta(z) \beta(y)\inv  \beta(x) \ ,
\end{equation}
showing that $U_{ab} = V_a \setminus \ker(\beta)$ is stable under the
ternary law. The idempotent laws follow by an easy computation from the
lemma. For para-associativity, using (\ref{TorsorFormula2}), a straightforward
computation shows that both $((xyz)uv)$ and $(x(uzy)v)$ are given by
\[
\beta(v)\beta(u)\inv \beta(z)\beta(y)\inv (x-y) +
\beta(v) \beta(u)\inv (u-z) + v .
\]

\ssk
(b') Remark: there is a slightly different version of (b), having the advantage
that the cases $a=b$ and $a \not= b$ can be treated simultaneously, and the
drawback that the dependence on $y$ is not visible:  choose $o:=y$ as origin in $V=V_a$,
and a linear form $\beta:V \to \bK$ such that
$b \cap V =\{ x \in V \mid \beta(x)=1 \}$.
The case $a=b$ then corresponds to $\beta = 0$.
A computation similar as above yields
\begin{equation}\label{TorsorFormula3}
xz = (xyz)_{ab}  = (1-\beta(z)) x + z = x - \beta(z) x + z
\end{equation}
from which associativity of the product $xz$ follows easily.
Note that Formula (\ref{TorsorFormula3}) is a special case of the formulae
given in Section 1.4 of \cite{BeKi10}.

\ssk
(c)
A third proof  follows from Theorem 2.3  in \cite{BeKi10}, combined with the
following general result:
\end{proof}

\begin{theorem}\label{LinkGammaLattice}
Let $\bK$ be a unital ring and $\cX=\cX(\Omega)$ be the full Grassmannian geometry
of some $\bK$-module $\Omega$ (set of all submodules of $\Omega$),
and define, for a $5$-tuple $(x,a,y,b,z) \in \cX^5$, the submodule
$\Gamma(x,a,y,b,z)$ by Equation  (\ref{GammaDef}).
\begin{enumerate}
\item
Assume that the triple $(x,y,z)$ is in \emph{general position}, that is,
\[
x \land (y \lor z)= 0, \quad \mbox{or } \quad
y \land (x \lor z)= 0, \quad \mbox{or } \quad
z \land (x \lor y)=0 .
\]
Then we have the following equality of submodules of $\Omega$:
\[
\Gamma(x,a,y,b,z) = \Bigl( \bigl( (x \lor y) \land a \bigr) \lor z \Bigr)
\land
\Bigl( \bigl( ( z \lor y) \land b \bigr) \lor x \Bigr) \ .
\]
\item
Assume that $z$ is contained in $x \lor y$, i.e., $z \land (x \lor y)= z$.
Then, for any choice of $u \in U_{ab}$ satisfying $u \land (x \lor y)=0$,
we have
\begin{eqnarray*}
\Gamma(x,a,y,b,z) &=&
([([ \Bigl( \bigl( (x \lor y) \land a \bigr) \lor u \Bigr)
\land
\Bigl( \bigl( ( u \lor y) \land b \bigr) \lor x \Bigr) ]
\lor u ) \land a ] \lor z ) \ \
\cr
&  & \land \quad    [(( z \lor u) \land b) \lor
 \Bigl( \bigl( (x \lor y) \land a \bigr) \lor u \Bigr)
\land
\Bigl( \bigl( ( u \lor y) \land b \bigr) \lor x \Bigr) ]
\end{eqnarray*}
\item
Let $a,b$ be hyperplanes in a vector space and $x,y,z$ lines. Retain
assumptions from the preceding item and assume that $x \not= y$. Then
the expression given there simplifies to
\begin{eqnarray*}
\Gamma(x,a,y,b,z) &=&
(x \lor y)  \land  [(( z \lor u) \land b) \lor
 \Bigl( \bigl( (x \lor y) \land a \bigr) \lor u \Bigr)
\land
\Bigl( \bigl( ( u \lor y) \land b \bigr) \lor x \Bigr) ] \ .
\end{eqnarray*}
\end{enumerate}
\end{theorem}

\begin{proof} (1)
We prove first the inclusion ``$\subset$'' (which holds in fact for all
triples $(x,y,z)$): on the one hand,
$ \Bigl( \bigl( (x \lor y) \land a \bigr) \lor z \Bigr) \land
\Bigl( \bigl( ( z \lor y) \land b \bigr) \lor x \Bigr)$
is the set of all $\omega \in \Omega$ such that we can write
\[
\omega = \alpha + \zeta, \qquad \omega = \beta + \xi
\]
with $\alpha \in a$, $\beta \in b$, which in turn can be written
\[
\alpha = \xi' +\eta\,, \qquad \beta = \zeta' + \eta'
\]
with $\xi' \in x$, etc.
This gives us a system (S) of 4 equations.

On the other hand, by definition, $\Gamma(x,a,y,b,z)$ is the set of all
$\omega \in \Omega$ such that
\[
\exists \xi \in x,
\exists \alpha \in a,
\exists \eta \in y,
\exists \beta \in b,
\exists \zeta \in z : \quad
\omega = \zeta + \alpha
= \alpha + \eta + \beta
= \xi + \beta
\]
There are several equivalent versions of this system (R) of three
equations -- see \cite{Be12}, Lemma 2.3., from which it is read off that
the four conditions from (S) are satisfied for $\omega \in \Gamma(x,a,y,b,z)$
if we choose $\xi'=\xi$, $\eta'=\eta$, $\zeta'=\zeta$.
Thus the inclusion ``$\subset$'' holds always.

The other inclusion does not always hold, but the theorem gives a sufficient
condition: indeed, if $\omega$ belongs to the set on the right hand side,
then (S) implies
\[
\omega = \xi' + \eta + \zeta = \zeta' + \eta' + \xi\,,
\]
whence $\xi - \xi' \in y \lor z$. If $x \land (y \lor z)=0$, this implies
that $\xi = \xi'$, and three of the four equations from (S) are equivalent
to (R). If $y \land (x \lor z)= 0$ or $z \land (x \lor y)=0$, then the same
argument applies (with respect to another choice of three from the four
equations of (S)). In all cases, it follows that
$\omega \in \Gamma(x,a,y,b,z)$.

(2) From \cite{BeKi10}, Theorem 2.3, we know that $\Gamma$ is para-associative
and satisfies the idempotent law:
\[
\Gamma\bigl( \Gamma(x,a,y,b,u),a,u,b,z \bigl) = \Gamma(x,a,y,b,\Gamma(u,a,u,b,z)) =
\Gamma(x,a,y,b,z)\,.
\]
By assumption, the triple $(x,y,u)$ is in  general position, and from this
it follows that the triple $((xyu),u,z)$ is also in general position;
therefore the left-hand side may be expressed in terms of the lattice structure
by applying twice part (1), which leads to the expression from the claim: in a
first step, we get
\[
 (xyu) =  \Bigl( \bigl( (x \lor y) \land a \bigr) \lor u \Bigr)
\land
\Bigl( \bigl( ( u \lor y) \land b \bigr) \lor x \Bigr) \,,
\]
and in a second step
\[
((xyu)uz) =
(([( \Bigl( \bigl( (x \lor y) \land a \bigr) \lor u \Bigr)
\land
\Bigl( \bigl( ( u \lor y) \land b \bigr) \lor x \Bigr) ]
\lor u ) \land a ] \lor z )
\]
\[
\land \ \
[(( z \lor u) \land b) \lor
 \Bigl( \bigl( (x \lor y) \land a \bigr) \lor u \Bigr)
\land
\Bigl( \bigl( ( u \lor y) \land b \bigr) \lor x \Bigr) ] \,.
\]
(3)
Under the given assumptions, the first term on the right hand side in (2)
reduces to the line $x \lor y$, and hence the claim follows directly from (2).
\end{proof}

\begin{remarks}
(a) Not all possible relative positions of
$(x,y,z)$ are covered by Theorem  \ref{LinkGammaLattice}, that is, the lattice
theoretic formula for $\Gamma(x,a,y,b,z)$ does not hold for all triples of submodules of $\Omega$.
For instance, if $\Omega = \bK^{2n}$ and $x,y,z$ are of dimension $n$,
then they cannot be in general position, and in general no $u$ as in (2) exists.
This case illustrates the special r\^ole of ``generalized projective lines'' (cf.\ \cite{BeKi10}) with
respect to lattice approaches.

\ssk
(b) Both for the definitions given here and in \cite{BeKi10}, it is not strictly necessary
that $x,y,z$ belong to $U_{ab}$: they may belong to $V_a$, or to $V_b$, or (in \cite{BeKi10})
be completely arbitrary. We will not enter here into a discussion of the relation of both
definitions if $x,y$ or $z$ does not belong to $U_{ab}$.

\ssk
(c) Both approaches lead to their own notions of \emph{morphisms}.
In the situation of Part (3) of Theorem \ref{LinkGammaLattice}, both of these notions
must lead to the same result:  this is precisely the
famous  ``second fundamental theorem of projective geometry''.
\end{remarks}

\section{The Moufang case}

\begin{theorem}\label{MoufangTheorem}
Assume $\cX$ is a Moufang projective plane and $(a,b)$ a pair of lines.
Then $U_{ab}$, together with the ternary product $(xyz)_{ab}$ defined in
the first section, is a \emph{ternary Moufang loop}. 
In particular, if we  fix an element  $y \in U_{ab}$ as origin, then
$U_{ab}$ with product $xz=(xyz)_{ab}$ and origin $y$ becomes a (binary)
Moufang loop.
\end{theorem}

\nin Before proving the theorem, we recall the relevant definitions
(cf.,  e.g., \cite{SBGHLS}):

\begin{definition}
A projective plane $\cX$ is a \emph{Moufang plane} if it satisfies
one of the following equivalent conditions
\begin{enumerate}
\item
The group of automorphisms fixing all points of any given line acts
transitively on the points not on the line.
\item
The group of automorphisms acts transitively on quadrangles.
\item
Any two ternary rings of the plane are isomorphic.
\item
Some ternary ring of the plane is an alternative division algebra, \emph{i.e.},
it is a division algebra satisfying the following identities:
\[
x(xy) = (xx)y\,,\qquad  (yx)x = y(xx)\,,\qquad (xy)x = x(yx)\,.
\]
\item
$\cX$ is isomorphic to the projective plane over an alternative division ring.
\item
The ``small Desargues theorem'' holds in all affine parts of $\cX$.
\end{enumerate}
\end{definition}

\nin
The set of invertible elements in alternative algebra forms a \emph{Moufang loop}.
A basic reference for loops in general and Moufang loops in particular is \cite{Bru71}.

\begin{definition}
A \emph{loop} $(Q,\cdot)$ is a set $Q$ with a binary operation
$Q^2 \to Q; (x,y) \mapsto xy$ such that for each $x$, the maps $y\mapsto xy$ and
$y\mapsto yx$ are bijections of $Q$, and having an element $e$ such that $ex=xe=x$
for all $x\in Q$.

A \emph{Moufang loop} is a loop $Q$ that satisfies any, and hence all of the
following equivalent identities (the \emph{Moufang identities}):
\begin{align*}
z(x(zy)) &= ((zx)z)y \tag{M1} \\
x(z(yz)) &= ((xz)y)z \tag{M2} \\
(zx)(yz) &= (z(xy))z \tag{N1} \\
(zx)(yz) &= z((xy)z) \tag{N2}
\end{align*}
\end{definition}

The \emph{left} and \emph{right multiplication maps} (sometimes called
translations) in a loop are defined, respectively by $L_x y := xy =: R_y x$.
The Moufang identities can be written in terms of the left and right
multiplication maps. For instance, the first two identities state that
\[
L_zL_xL_z  = L_{zxz}\myand\qquad R_zR_yR_z  = R_{zyz}\,.
\]
\emph{Moufang's Theorem} implies that Moufang loops are \emph{diassociative},
that is, for any $a,b$, the subloop $\langle a,b\rangle$ generated by $a,b$ is
a group. This can be seen as a loop theoretic analog of Artin's Theorem for
alternative algebras. Two particular instances of diassociativity are the
\emph{left} and \emph{right inverse properties}
\begin{align*}
x\inv (xy) &= y \tag{LIP} \\
(xy)y\inv &= x\,, \tag{RIP}
\end{align*}
where $x\inv$ is the unique element satisfying $xx\inv = x\inv x = e$.
The following lemma gives the Moufang analog of the well-known relation
between torsors and groups:

\begin{lemma}\label{MoufangLemma}
Let $Q$ be a Moufang loop, and define a ternary operation
$(\cdot\cdot\cdot): Q^3 \to Q$ by $(xyz):=(xy\inv)z$. Then the following
three identities hold:
\begin{align*}
(xxy) = &\ y = (yxx) \tag{MT0} \\
(uv(xyx)) &= ((uvx)yx) \tag{MT1} \\ 
(xy(xyz)) &= ((xyx)yz)  \tag{MT2}   
\end{align*}

Conversely, if $M$ is a set with a ternary operation
$(\cdot\cdot\cdot):M^3 \to M$ satisfying (MT0), (MT1) and (MT2),
then, for every choice of ``origin'' $e \in M$, the binary operation
$x\cdot y:=(xey)$ and the unary operation $x\inv := (exe)$ define the structure
of a Moufang loop on $M$ with neutral element $e$.
\end{lemma}

\begin{proof}
Firstly assume $Q$ is a Moufang loop. The leftmost identity in (MT0) is trivial
while the rightmost follows immediately from (RIP). For (MT1), we compute
\begin{align*}
(uv(xyx)) &= (uv\inv)((xy\inv)x) & \\
&= (uv\inv)(x(y\inv x)) & \text{(}\langle x,y\rangle\text{ is a group)} \\
&= (((uv\inv)x)y\inv)x & (M2) \\
&= ((uvx)yx)\,. \\
\intertext{For (MT2),}
(xy(xyz)) &= (xy\inv)((xy\inv)z) & \\
&= ((xy\inv)(xy\inv))z & \text{(}\langle xy\inv,z\rangle\text{ is a group)} \\
&= (((xy\inv)x)y\inv)z & \text{(}\langle x,y\rangle\text{ is a group)} \\
&= (((xyx)y)z)\,.
\end{align*}

Conversely, suppose $M$ is a set with a ternary operation
$(\cdot\cdot\cdot):M^3 \to M$ satisfying (MT0), (MT1) and (MT2). Fix
$e\in M$ and define $x\cdot y:=(xey)$ and $x\inv := (exe)$ for all
$x,y\in M$. By (MT0), we see that $e$ is neutral element for the binary
operation.

First we establish the following identities:
\begin{align}
x\cdot y\inv &= (xye)\,, \eqnlabel{tmp15} \\
(x\cdot y\inv)\cdot z\inv &= (xyz\inv)\,, \eqnlabel{tmp21} \\
(x\inv)\inv\cdot x &= e\,, \eqnlabel{tmp19} \\
((x\inv)\inv xy\inv ) &= y\,. \eqnlabel{tmp23}
\end{align}
For \eqnref{tmp15} we compute $x\cdot y\inv = (xe(eye)) = ((xee)ye) = (xye)$ using (MT1) in the second
equality and (MT0) in the third. For \eqnref{tmp21}, we have
$(x\cdot y\inv)\cdot z\inv = ((xye)ze) = (xy(eze)) = (xyz\inv)$, using
\eqnref{tmp15} (twice) and (MT1). For \eqnref{tmp19},
$(x\inv)\inv\cdot x = ((x\inv)\inv xe) = ((ex\inv e)xe)
= (ex\inv (exe)) = (ex\inv x\inv) = e$, using \eqnref{tmp15} in the first
equality, (MT1) in the third and (MT0) in the fourth. Finally, for
\eqnref{tmp23}, $((x\inv)\inv xy\inv ) = ((x\inv)\inv\cdot x\inv)\cdot y\inv
= e\cdot y\inv = y\inv$, using \eqnref{tmp21} followed by \eqnref{tmp19}.

Next we prove
\begin{equation}
\eqnlabel{tmpnear}
(xy(y\inv)\inv) = x\,.
\end{equation}
Indeed,
\begin{alignat*}{2}
(xy(y\inv)\inv) &= ((x(y\inv)\inv (y\inv)\inv)y (y\inv)\inv)
&&= (x(y\inv)\inv ((y\inv)\inv y (y\inv)\inv)) \\
&= (x(y\inv)\inv (y\inv)y\inv) &&= x\,,
\end{alignat*}
where we have used (MT0), (MT1), \eqnref{tmp23} and (MT0).

Taking $y = x$ in \eqnref{tmpnear} and applying (MT0), we obtain
\begin{equation}
\eqnlabel{tmpx''}
(x\inv)\inv = x\,.
\end{equation}
From this it follows that $e\inv = e$, since $e\inv = e\cdot e\inv
= (e\inv)\inv \cdot e\inv = e$.

Now in \eqnref{tmp21}, replace $z$ with $z\inv$ and use \eqnref{tmpx''} to obtain
\begin{equation}
\eqnlabel{tmp21a}
(x\cdot y\inv)\cdot z = (xyz)\,.
\end{equation}
Replacing $y$ with $y\inv$ and then setting $z = y\inv$ in \eqnref{tmp21a}, we
obtain $(x\cdot y)\cdot y\inv = (xy\inv y\inv) = x$ using (MT0). Thus the
right inverse property (RIP) holds.

Next we almost obtain the Moufang identity (M2) as follows:
\begin{align*}
((x\cdot y)\cdot z)\cdot y &= (((x\cdot e)\cdot y)\cdot z)\cdot y \\
&= ((xey)z\inv y) \\
&= (xe(yz\inv y)) \\
&= (x\cdot e)\cdot ((y\cdot z)\cdot y)\,,
\end{align*}
using \eqnref{tmp21a}, (MT1) and  \eqnref{tmp21a} again.
In loop theory, this is known as the \emph{right Bol identity}.

We also have the left alternative law:
\begin{align*}
(x\cdot x)\cdot y &= (((x\cdot e)\cdot x)\cdot e)\cdot y \\
&= ((xex)ez) \\
&= (xe(xez)) \\
&= (x\cdot e)\cdot ((x\cdot e)\cdot y) \\
&= x\cdot (x\cdot y)\,,
\end{align*}
using \eqnref{tmp21a}, (MT2) and  \eqnref{tmp21a} again.

The rest of the argument is standard. A magma satisfying the right
Bol identity and (RIP) is a loop, called a \emph{right Bol loop}
(see, \emph{e.g.}, \cite{Kie02}, Theorem 3.11, suitably dualized).
A right Bol loop satisfying the left alternative law is a Moufang
loop \cite{Rob66}.
\end{proof}

\begin{definition}
A set $M$ with a map $(\cdot\cdot\cdot) : M^3 \to M$ satisfying the three
identities from Lemma  \ref{MoufangLemma} will be called a
\emph{ternary Moufang loop}. 
\end{definition}

\begin{remarks}
(1) The axioms (MT1) and (MT2) for ternary Moufang loops are precisely the
identities (AP2) and (AP3) in Loos' axiomatization of an alternative pair
\cite{Lo75}.

\ssk
(2) For an associative torsor $(\cdot\cdot\cdot): M^3\to M$, the groups
determined by different choices of ``origin,'' that is, fixed middle slot,
are all isomorphic. The analog of this does not hold for ternary Moufang
loops. Instead the different Moufang loops are \emph{isotopic} \cite{Bru71}.
In fact, it is straightforward to show that for a Moufang loop $Q$, each
isotope of $Q$ is isomorphic to an isotope with multiplication given by
$x\circ z = (xy\inv)z$ for some $y\in Q$. Thus just as alternative
triple systems encode all homotopes of an alternative algebra into a
single structure, so do ternary Moufang loops encode all isotopes of
Moufang loops.

\ssk
(3) Though we did not bother to state this in the lemma, it is clear from
the proof that if we start with a Moufang loop $Q$ with neutral element $e$,
construct the corresponding ternary operation $(\cdot\cdot\cdot)$ and then
construct the binary and unary operations induced by
$(\cdot\cdot\cdot)$ with origin $e$, we recover the original loop
operations. Similarly, if we start with a ternary Moufang loop $M$,
construct  the binary and unary operations with origin $e$
and then construct the corresponding ternary operation induced by the
loop structure, we recover the original ternary Moufang loop.
\end{remarks}

\begin{proof}[Proof of Theorem \ref{MoufangTheorem}.]
In principle, the first two strategies of proof of Theorem \ref{TorsorDesarguesTheorem}  carry over:

\ssk
(a) A proof in the framework of axiomatic geometry.
Instead of the full Desargues theorem we now can only use the Little
Desargues theorem. The drawings will become more complicated than above
since one has to introduce auxiliary points. We will not pursue this proof
here.

\ssk
(b) A computational proof.
Let $\bK$ be the alternative division ring belonging to the plane.
Then the affine space $V:= V_a$ is isomorphic to $\bK^2$,
and affine lines can be described as in the Desarguesian case, eg.
$x \lor y = \{ (1-t)x + ty \mid t \in \bK \}$, where multiplication by ``scalars''
in $\bK^2$ is componentwise.
If $a=b$, then  $(xyz)_{aa} = x - y +z$ is the torsor law of the abelian group
 $V_a \cong (\bK^2,+)$, and the claim is  true.
If $a \not= b$, fix some arbitrary origin $o$ in the affine hyperplane $V_a \cap b$.
There is a linear form $\beta:  V \to \bK$ such that $b \cap V =\ker(\beta)$, so that
$U_{ab} = \{ x \in V_a \mid \beta(x) \in \bK^\times \}$.
(To fix things, one may choose coordinates such that $\beta = \pr_1$
is the projection onto the
first coordinate of $\bK^2$, so $b$ is the vertical axis.)

\begin{lemma}\label{ComputationLemma2}
Let notation be as above. Then, for all $x,y,z \in U_{ab}$, we have
\[
(xyz)_{ab} = ( \beta(z)\beta(y)\inv ) \cdot (x-y) + z .
\]
\end{lemma}

\begin{proof}
The proof of Lemma \ref{ComputationLemma} carries over without any changes --
associativity of the ring has not been used there, only some elementary properties
of inverses which are direct consequences of the left and right inverse properties
(LIP) and (RIP).
\end{proof}

From the lemma we get, as before,  the formula
\begin{equation}\label{TorsorFormula4}
\beta \bigl( (xyz) \bigr) =
[\beta(z) \beta(y)\inv ] \beta(x)
= \bigl( \beta(z) \beta(y) \beta(x )\bigr)
\end{equation}
which means that $\beta$ induces a homomorphism from $U_{ab}$ to the
ternary Moufang loop  $\bK^\times$.

This formula is crucial in the proof of the alternative laws of $U_{ab}$:
essentially, it implies that identities holding in
$\bK^\times$ will carry over to $U_{ab}$; but
 the unit loop of $\bK$ is a ternary Moufang loop, and hence so will be $U_{ab}$.
For instance, for the proof of (MT1), $(uv(xyx)) = ((uvx)yx)$,
write both sides, using the lemma: one sees that equality holds iff,
for the vector $w:=u-v \in \bK^2$ and for all $x,y,v \in \bK^2 \setminus \ker(\beta)$,
we have
\[
((( \beta(x) \beta(y)\inv) \beta(x) ) \beta(v)\inv ) w =
(\beta(x)\beta(y)\inv) ((\beta(x)\beta(v)\inv)w)
\]
But this amounts to an identity in $\bK$ (or, if one prefers,
two identities, one for each component of $w$),  of the same form as the one we want
to prove; this identity holds since $\bK$ is an alternative algebra.
\end{proof}

\section{Prospects}

In subsequent work, we will investigate more thoroughly the geometry corresponding
 to alternative algebras and
alternative pairs (cf.\ \cite{Lo75}): ``alternative geometries'' correspond to  such algebras
in a similar way as the associative geometries from
\cite{BeKi10} correspond to associative algebras and associative pairs.
They play a key r\^ole in the construction of exceptional spaces
corresponding to Jordan algebas and Jordan pairs.
In the following, we briefly mention some topics to be discussed in this context.

\subsection{Structure of the torsors and ternary Moufang loops}
First of all, it is easy to understand the structure of the groups $U_{ab}$ in the Desarguesian case:
for $a=b$,  $U_{ab}=V_a$ is a vector group (this is true even in the Moufang case),
and for $a \not=b$,
 $U_{ab}$ is isomorphic to the \emph{dilation} or \emph{$ax+b$-group}
\begin{equation}\label{DilEqn}
\Dil(E) :=\{ f:E \to E \mid \, f(x)=ax+b, \ b \in E, a \in \bK^\times \}
\end{equation}
of the affine space $E=a_b=a \setminus b$ (where $a \cap b$ is considered as
hyperplane at infinity of $a$).
This dilation group, in turn, is a semidirect product of $\bK^\times$ with the translation group
of $E$. The resulting homomorphism $U_{ab} \to \bK^\times$ can be described in a purely
geometric way (cf.\
\cite{Be12}, Theorem 7.4 for the case of very general Grassmannians).
For Moufang planes, partial analogs of this hold:
 there is a split exact sequence  of ternary Moufang loops
\[
\begin{matrix}
a_b  & \to &U_{ab} & \to & \bK^\times \  ,
\end{matrix}
\]
where any line $L$ in $\cX$ which intersects $a \cup b$ in exactly two different
points provides a splitting.
 But, if the plane is not Desarguesian,
 the set $\Dil(E)$ defined by
(\ref{DilEqn}) is then no longer a group, nor is it contained in the automorphism group
of the plane.
However, it remains true in the Moufang case that one obtains a \emph{symmetric plane}
(defined in \cite{Loe}; in the rough classification of symmetric planes by H.\ L\"owe \cite{Loewe01},
our spaces appear among the \emph{split symmetric planes}.)

\subsection{Duality}
Carrying out our geometric construction from Chapter 1 in the \emph{dual projective space},
by  general duality principles of projective geometry, we get again torsors,
respectively ternary Moufang loops.
Remarkably, the description of the torsors in the Desarguesian case by
equation (\ref{GammaDef}) does not change, except for a switch in $a$ and $b$.
In other words, up to this switch, the map $\Gamma$ is ``self-dual'', which is in keeping
with results on anti-automorphisms from \cite{BeKi10b}. For the moment, it is an open
problem whether a similar ``self-dual description'' exists also in the Moufang case.

\subsection{General projective planes}
Our definition of $(xyz)_{ab}$ in the generic case (Definition \ref{GenericDefinition}) makes
sense for any projective plane (and even for any lattice if we admit $0$ as possible result).
What, then, are its properties?
In particular,
what is its relation with  the ``ternary field'' associated to a quadruple of points in the plane?
Put differently, how do we have to modify the definition in the collinear case
(Definition \ref{CollinearDefinition})?
Does the ``split exact sequence'' $a_b   \to U_{ab}  \to  \bK^\times$
survive in some suitable algebraic category?
Can one  re-interprete the classical Lenz-Barlotti types of projective planes
(cf., e.g., \cite{SBGHLS}, p.\ 142) in terms of $(xyz)_{ab}$?


\subsection{Perspective drawing}
Our construction also has aspects that should be interesting for applied sciences:
as already pointed out, our two-dimensional drawings have a ``spacial interpretation''.
This can be explained by observing that the torsors $U_{AB}$ living in a
three-dimensional space $\bK \bP^3$ can be mapped homomorphically onto torsors
$U_{ab}$ living in a projective plane $P$ (by choosing $P \subset \bK\bP^3$ intersecting $A \land B$ in
a single point and projecting from a point $q \in A \land B$, $q \notin P$, onto $P$;
then let $a:=P \land A$ and $b:=P \land B$).
A careful look shows that the torsor structure thus represented on $P$ is quite often implicitly used
in two-dimensional ``perspective representations'' of three-dimensional space;
however, to our knowledge, the underlying algebraic structure has so far not yet been
clearly recognized.

\end{document}